\documentclass{amsart}
\usepackage{amsfonts,amssymb,amsmath,amsthm}
\usepackage{url}
\usepackage{enumerate}

\newtheorem{thrm}{Theorem}[section]
\newtheorem{lem}[thrm]{Lemma}

\newtheorem{cor}[thrm]{Corollary}
\newtheorem{definition}[thrm]{Definition}

\newtheorem{remark}[thrm]{Remark}
\newtheorem{ex}[thrm]{Example}
\numberwithin{equation}{section}

\author{C.A. Morales}

\address{Instituto de Matem\'atica\\
Universidade Federal do Rio de Janeiro\\
P. O. Box 68530 21945-970\\
Rio de Janeiro\\
Brazil.}
\email{morales@impa.br}

\begin{document}

\title[Shadowable points]{Shadowable points}

\begin{abstract}
We define shadowable points for homeomorphism on metric spaces.
In the compact case
we will prove the following results: The set of shadowable points is invariant, possibly nonempty or noncompact.
A homeomorphism has the pseudo-orbit tracing property if and only if every
point is shadowable. The chain recurrent and nonwandering sets coincides
when every chain recurrent point is shadowable. 
Minimal or distal homeomorphisms of compact connected metric spaces have no shadowable points.
The space is totally disconnected at every shadowable point for distal homeomorphisms (and conversely for equicontinuous homeomorphisms).
A distal homeomorphism has the pseudo-orbit tracing property if and only if
the space is totally disconnected (this improves Theorem 4 in \cite{mo}).
\end{abstract}

\maketitle

\section{Introduction}
\noindent
Let $f: X\to X$ be a homeomorphism of a metric space $X$.
If $\delta>0$ we say that a bi-infinite sequence $\xi=(\xi_n)_{n\in \mathbb{Z}}$ of $X$
is a {\em $\delta$-pseudo-orbit} if $d(f(\xi_n),\xi_{n+1})\leq\delta$ for all $n\in\mathbb{Z}$.
We say that  $\xi$ {\em can be $\delta$-shadowed}
if there is $x\in X$ such that $d(f^n(x),\xi_n)\leq \delta$ for all $n\in\mathbb{Z}$. 
We say that $f$ has the {\em pseudo-orbit tracing property} (abbrev. {\em POTP}) if
for every $\epsilon>0$ there is $\delta>0$ such that
every $\delta$-pseudo-orbit can be $\epsilon$-shadowed.
Homeomorphisms with the POTP have been widely studied \cite{ah}, \cite{p0}.

In this paper we will study the following concept closely related to that of
absolutely nonshadowable points \cite{yy}. It splits
the POTP into individual shadowings.

\begin{definition}
A point $x\in X$ is {\em shadowable} if for every $\epsilon>0$ there is $\delta>0$
such that every $\delta$-pseudo-orbit $\xi$ {\em with} $\xi_0=x$
can be $\epsilon$-shadowed.
\end{definition}

Here are some examples, where $Sh(f)$ denote the set of shadowable points of $f$.

\begin{ex}
\label{ex1}
Clearly if $f$ has the POTP, then $Sh(f)=X$ (i.e. every point is shadowable).
The converse is true on compact metric spaces by the next theorem.
As we shall see, the identity of the circle has no shadowable points.
Examples where $Sh(f)$ is a proper nonempty set will be given later on.
\end{ex}

We give some properties of $Sh(f)$ through the following standard definitions.
We say that a point $x\in X$ is
{\em nonwandering} if for every neighborhood $U$ of $x$ there is $k\in\mathbb{N}^+$ such that
$f^n(U)\cap U\neq\emptyset$. We say that $x$ is
{\em chain recurrent} if for every $\rho>0$ there is a {\em $\rho$-chain} from
$x$ to itself, i.e., a finite sequence
$\{x_i: 0\leq i\leq n\}$ satisfying $x_0=x$, $x_n=y$ and $d(f(x_i),x_{i+1})\leq\rho$
for all $i$ with $0\leq i\leq n-1$.
Denote by $\Omega(f)$ and $CR(f)$ the set of nonwandering and chain recurrent points of $f$ respectively.
We say that $\Lambda\subset X$ is {\em invariant} if $f(\Lambda)=\Lambda$.

With these definitions we can state our first result.

\begin{thrm}
\label{thA}
Let $f$ be a homeomorphism of a compact metric space $X$.
\begin{enumerate}
\item
$Sh(f)$ is an invariant set (possibly nonempty or noncompact).
\item
$f$ has the POTP if and only if $Sh(f)=X$.
\item
If $CR(f)\subset Sh(f)$, then $CR(f)=\Omega(f)$.
\end{enumerate}
\end{thrm}

Recall that a homeomorphism $f:X\to X$ is {\em minimal}
if the orbit $\{f^n(x):n\in\mathbb{Z}\}$ of every point $x\in X$ is dense in $X$.
It is well known that a minimal homeomorphism of a compact connected metric space with more than one point
does not have the POTP \cite{a}, \cite{m}.
Adapting the arguments in \cite{a} we can prove the following.

\begin{thrm}
\label{thC}
A minimal homeomorphism of a compact connected metric space with more than one point has no shadowable points.
\end{thrm}

Recall also that  $f$ is
{\em distal} if $\inf_{n\in\mathbb{Z}}d(f^n(x),f^n(y))>0$ for all distinct points $x,y\in X$. We
say that $X$ is totally disconnected at $p\in X$ if the connected component of $X$ containing $p$ is $\{p\}$.
As in \cite{bdrss} we define
$$
X^{deg}=\{p\in X:X\mbox{ is totally disconnected at }p\}.
$$
With these notations we obtain the following result.
\begin{thrm}
\label{thB}
If $f:X\to X$ is distal homeomorphism of a compact metric space $X$, then $Sh(f)\subset X^{deg}$.
\end{thrm}

Let us state some short corollaries of Theorem \ref{thB}.
Likewise the minimal, the distal
homeomorphisms of a compact connected metric space with more than one point
do not have the POTP \cite{a}, \cite{m}.
This motivates the following result.

\begin{cor}
\label{thD}
A distal homeomorphism of a compact connected metric space with more than one point
has no shadowable points.
\end{cor}

\begin{proof}
Clearly, a connected space with more than one point has no totally disconnected points.
Then, there are no shadowable points for distal homeomorphisms on such a space too by Theorem \ref{thB}.
\end{proof}

The second corollary deals with compact metric spaces exhibiting distal homeomorphisms with the POTP.
Recall that $X$ is {\em totally disconnected} if $X=X^{deg}$.
Every totally disconnected compact metric spaces exhibits a distal homeomorphism with the POTP
(e.g. the identity, see Theorem 2.3.2 p. 79 in \cite{ah}). Conversely, we obtain the following result.

\begin{cor}
\label{c1}
Every compact metric space exhibiting a distal homeomorphism with the POTP is totally disconnected.
\end{cor}

\begin{proof}
The existence of a distal homeomorphism with the POTP implies that every point is
shadowable. Then, the space is totally disconnected by Theorem \ref{thB}.
\end{proof}

It is worth to note that
every distal homeomorphism with the POTP of a compact metric space
is uniformly conjugate to an adding-machine-like map \cite{my}.
Theorem \ref{thB} also implies the following result.
We say that a homeomorphism $f: X\to X$ has the {\em almost POTP} if $Sh(f)$ is dense in $X$.
As in Definition 1 of \cite{bdrss}
we say that the space $X$ is {\em almost totally disconnected} if $X^{deg}$ is dense in $X$.

\begin{cor}
\label{almost}
Every compact metric space $X$ exhibiting a distal homeomorphism with the almost POTP
is almost totally disconnected.
\end{cor}

On the other hand, Theorem \ref{thB} motivates the question if
$Sh(f)=X^{deg}$ for all distal homeomorphisms $f:X\to X$.
Partial positive answers on compact connected metric spaces with more than one point are given by Corollary \ref{thD}; or
by Theorem \ref{thC} in the transitive case (because, in such a case, $f$ is minimal). In these cases we get $Sh(f)=X^{deg}=\emptyset$.

Another partial positive answer is as follows.
We say that a homeomorphism $f: X\to X$ is {\em equicontinuous}
if for every $\alpha>0$ there is $\beta>0$ such that
$x,y\in X$ and $d(x,y)\leq \beta$ imply
$d(f^n(x),f^n(y))\leq \alpha$ for all $n\in\mathbb{Z}$.
It is easy to see that every equicontinuous homeomorphism of a compact metric space is distal.
For such homeomorphisms we have the following result.

\begin{thrm}
\label{separated}
If $f:X\to X$ is a equicontinuous homeomorphism of a compact metric space $X$, then
$Sh(f)=X^{deg}$.
\end{thrm}

From this we obtain the following corollary extending the conclusion of Theorem 4 in \cite{mo} to distal homeomorphisms.

\begin{cor}
\label{thathu}
Let $X$ be a compact metric space and $f : X \to X$ be a distal homeomorphism. Then, $f$ has the POTP
if and only if $X$ is totally disconnected.
\end{cor}

\begin{proof}
If $f$ has the POTP, then $X$ is totally disconnected by Corollary \ref{c1}.
Conversely, if $X$ is totally disconnected, then $f$ is equicontinuous (e.g. Corollary 1.9 in \cite{agw}) so
$Sh(f)=X$ (by Theorem \ref{separated}) thus $f$ has the POTP (by Theorem \ref{thA}).
\end{proof}

In particular, we obtain the following result supporting Corollary \ref{almost}.

\begin{cor}
There is a compact metric space exhibiting a distal homeomorphism with the almost POTP
but without the POTP.
\end{cor}

\begin{proof}
Take an almost totally disconnected compact metric space $X$ which is not totally disconnected
(e.g \cite{f} or a cantoroid as in Definition 2 p. 70 of \cite{bdrss}).
Then, the identity $f: X\to X$ (which is equicontinuous)
has the almost POTP (by Theorem \ref{separated}) but not the POTP (by Corollary \ref{thathu}).
\end{proof}

The author would like to thank professors B. Carvalho and D. Obata for helpful conversations.

\section{Proof of the theorems}

\noindent
Let $X$ be a compact metric space.
We say that a sequence $(x_n)_{n\in\mathbb{Z}}$ of $X$ is {\em through} some subset $K\subset X$ if $x_0\in K$.
We shall use the following auxiliary definition.

\begin{definition}
We say that a homeomorphism $f:X\to X$ has the {\em POTP through $K$}
if for every $\epsilon>0$ there is $\delta>0$ such that every
$\delta$-pseudo-orbit of $f$ through $K$ can be $\epsilon$-shadowed.
\end{definition}

This definition is stronger than the POTP {\em on} $K$
in which the shadowing is guaranteed for pseudo-orbits enterely contained in $K$ only \cite{p0} .
We shall use the following characterization in which
$B[\cdot,\delta]$ denotes the closed $\delta$-ball operation.

\begin{lem}
\label{l1}
A homeomorphism of a compact metric space has the POTP through a subset $K$
if and only if for every $\epsilon>0$ there is $\delta>0$ such
that every $\delta$-pseudo-orbit through $B[K,\delta]$ can be $\epsilon$-shadowed.
\end{lem}

\begin{proof}
Obviously we only have to prove the necessity.
Suppose by contradiction that a homeomorphism $f$ of a compact metric space $X$
has the POTP through $K$ but there are $\epsilon>0$ and a sequence
of $\frac{1}{k}$-pseudo-orbits $(\xi^k)_{k\in\mathbb{N}^+}$
through $B[K,\frac{1}{k}]$ which cannot be $2\epsilon$-shadowed.

For this $\epsilon$ we take $\delta$ from the POTP through $K$. We can assume $\delta<\epsilon$.
It follows from the definition that there is a sequence $x^k\in K$ such that $d(\xi^k_0,x^k)\leq\frac{1}{k}$ for all $k\in\mathbb{N}^+$.
As $X$ is compact, $f$ is uniformly continuous so we can fix $k$ large such that
$\max\{d(f(\xi^k_0),f(x^k)),\frac{1}{k}\}\leq \frac{\delta}{2}$.
Once we fix this $k$ we
define the sequence $\hat{\xi}=(\hat{\xi}_n)_{n\in\mathbb{Z}}$ by
$$
\hat{\xi}_n = \left\{
\begin{array}{rcl}
\xi^k_n,& \mbox{if} & n\neq0\\
x^k,  & \mbox{if} & n=0.
\end{array}
\right.
$$

Clearly $d(f(\hat{\xi}_n),\hat{\xi}_{n+1})\leq\frac{1}{k}\leq \delta$ for $n\neq -1,0$.
Since
$$
d(f(\hat{\xi}_{-1}),\hat{\xi}_0)=d(f(\xi^k_{-1}),x^k)
\leq d(f(\xi^k_{-1}),\xi^k_0)+d(\xi^k_0,x^k)\leq \frac{1}{k}+\frac{1}{k}=\frac{2}{k}\leq\delta
$$
and
$$
d(f(\hat{\xi}_0),\hat{\xi}_1)=d(f(x^k),\xi^k_1)\leq d(f(x^k),f(\xi^k_0))+d(f(\xi^k_0),\xi^k_1)\leq\frac{\delta}{2}+\frac{\delta}{2}=\delta
$$
we see that $\hat{\xi}$ is a $\delta$-pseudo-orbit.
Since $\hat{\xi}_0=x^k\in K$ by definition we obtain that
$\hat{\xi}$ can be $\epsilon$-shadowed, namely,
there is $y\in X$ such that $d(f^n(y),\hat{\xi}_n)\leq\epsilon$ for every $n\in\mathbb{Z}$.

Clearly $d(f^n(y),\xi^k_n)=d(f^n(y),\hat{\xi}_n)\leq \epsilon\leq 2\epsilon$
for $n\neq 0$.
For $n=0$ we obtain
$$
d(f^n(y),\xi^k_n)=d(y,\xi^k_0)\leq d(y,x^k)+d(x^k,\xi^k_0)
=d(y,\hat{\xi}_0)+\frac{1}{k}\leq \epsilon+\frac{\delta}{2}\leq 2\epsilon
$$
thus $d(f^n(y),\xi^k_n)\leq2\epsilon$ for all $n\in\mathbb{Z}$.
It follows that $\xi^k$
can be $2\epsilon$-shadowed, that is absurd.
This contradiction proves the result.
\end{proof}

The proof of the lemma below is similar to Lemma 1 in \cite{a}.

\begin{lem}
\label{l2}
Let $f$ be a homeomorphism of a compact metric space.
Then, for every $z\in \Omega(f)\cap Sh(f)$ and every $\epsilon>0$ there
are $k\in \mathbb{N}^+$ and $y\in X$ such that $f^{pk}(y)\in B[z,\epsilon]$ for every $p\in\mathbb{Z}$.
\end{lem}

\begin{proof}
Fix $z\in \Omega(f)\cap Sh(f)$ and $\epsilon>0$.
Let $\delta>0$ be given by Lemma \ref{l1} for $\frac{\epsilon}{2}$ with $K=\{z\}$. Obviously we can assume $\delta<\epsilon$.
Since $z\in \Omega(f)$, there are $x\in X$ and $k\in\mathbb{N}^+$ such that
$x,f^k(x)\in B[z,\frac{\delta}{2}]$.
Now consider the sequence $(x_n)_{n\in\mathbb{Z}}$ defined by
$x_{pk+r}=f^r(x)$ for $p\in \mathbb{Z}$ and $0\leq r<k$.
Obviously $(x_n)_{n\in\mathbb{Z}}$ is a $\delta$-pseudo-orbit with $x_0\in B[z,\delta]$, and so,
by Lemma \ref{l1}, there is $y\in X$ such that
$d(f^n(y),x_n)\leq\epsilon$ for every $n\in \mathbb{Z}$.
Taking $n=pk$ with $p\in\mathbb{Z}$ we obtain $d(f^{pk}(y),x)\leq \frac{\epsilon}{2}$, and so,
$d(f^{pk}(y),z)\leq d(f^{pk}(y),x)+d(x,z)\leq \frac{\epsilon}{2}+\frac{\epsilon}{2}=\epsilon$
for all $p\in\mathbb{Z}$.
\end{proof}

Clearly if $f$ has the POTP through $K$, then
every point in $K$ is shadowable.
The converse is true when $K$ is compact by the following result.

\begin{lem}
\label{naosei}
A homeomorphism of a compact metric space has the
POTP through a compact subset $K$ if and only if every point in $K$ is shadowable.
\end{lem}

\begin{proof}
By the previous remark we only have to prove the sufficiency.

Let $f: X\to X$ be a homeomorphism of a compact metric space $X$.
For this we assume by contradiction that there is a compact subset $K$ such that
every point in $K$ is shadowable but $f$ has no the POTP through $K$.
Then, there are $\epsilon>0$ and a sequence $(\xi^k)_{k\in\mathbb{N}^+}$ of $\frac{1}{k}$-pseudo-orbits through $K$
which cannot be $2\epsilon$-shadowed.

Since $K$ is compact,
we can assume that $\xi^k_0\to p$ for some $p\in K$.
Since $p\in K$, we have that $K$ is shadowable.
Then, we can take $\delta>0$ from the shadowableness of $p$ for the above $\epsilon$.
We define the sequence $\hat{\xi}^k=(\hat{\xi}^k_n)$ by
$$
\hat{\xi}^k_n = \left\{
\begin{array}{rcl}
\xi^k_n,& \mbox{if} & n\neq0\\
p,  & \mbox{if} & n=0,
\end{array}
\right.
\quad\quad k\in \mathbb{N}^+.
$$
Clearly all such sequences are through $\{p\}$.
Moreover,
$$
d(f(\hat{\xi}^k_{n-1}),\hat{\xi}^k_n) = \left\{
\begin{array}{rcl}
d(f(\xi^k_{n-1}),\xi^k_n),& \mbox{if} & n\neq0,1\\
d(f(p),\xi^k_1),& \mbox{if} & n=1 \\
d(f(\xi^k_{-1}),p),  & \mbox{if} & n=0
\end{array}
\right.
$$
so
$$
d(f(\hat{\xi}^k_{n-1}),\hat{\xi}^k_n) \leq \left\{
\begin{array}{rcl}
\frac{1}{k},& \mbox{if} & n\neq0,1\\
d(f(p),f(\xi^k_0))+\frac{1}{k},& \mbox{if} & n=1 \\
d(\xi^k_0,p)+\frac{1}{k},  & \mbox{if} & n=0.
\end{array}
\right.
$$
As $f$ is continuous and $\xi^k_0\to p$, we obtain that
$(\hat{\xi}^k_n)$ is a $\delta$-pseudo-orbit for $k$ large.
Then, for such a $k$ it follows that there is $x_k\in X$ such that
$d(f^n(x_k),\hat{\xi}^k_n)\leq \epsilon$ for every $n\in\mathbb{Z}$.
It follows that $d(f^n(x_k),\xi^k_n)\leq \epsilon$ for $n\neq 0$. Since
$$
d(x_k,\xi^k_0)\leq d(x_k,p)+d(p,\xi^k_0)\leq \epsilon+d(p,\xi^k_0)
$$
we have that $d(f^n(x_k),\xi^k_n)\leq 2\epsilon$ also for $n=0$ with $k$ large.
We conclude that $\xi^k$ {\em can be $2\epsilon$-shadowed for $k$ large}. This is a contradiction
which completes the proof.
\end{proof}

We observe that this lemma is false if $K$ were noncompact (by Remark \ref{rm1}).
Further properties of the shadowable points are given below.

\begin{lem}
\label{invariant}
The set of shadowable points of a homeomorphism of a compact metric space is invariant.
\end{lem}

\begin{proof}
It suffices to prove that if $x$ is a shadowable point
of a homeomorphism $f: X\to X$ of a compact metric space $X$,
then so are $f(x)$ and $f^{-1}(x)$. We only prove that $f(x)$ is shadowable as the same proof works for
$f^{-1}(x)$.

Fix $\epsilon>0$.
Since $X$ is compact, $f$ is uniformly continuous so
there is $\epsilon'>0$ such that
$d(f(y),f(z))\leq\epsilon$ whenever $y,z\in X$ satisfy $d(y,z)\leq \epsilon'$.
For this $\epsilon'$ we let $\delta'>0$ be given by the shadowableness of $x$.
Again $X$ is compact so $f^{-1}$ is uniformly continuous
thus there is $\delta>0$ such that
$d(f^{-1}(y),f^{-1}(z))\leq \delta'$ whenever $y,z\in X$ satisfy $d(y,z)\leq \delta$.

Now take a $\delta$-pseudo-orbit $(x_n)_{n\in\mathbb{Z}}$ through $f(x)$.
It follows from the choice of $\delta$ that $(f^{-1}(x_n))_{n\in\mathbb{Z}}$ is a $\delta'$-pseudo-orbit
which is obviously through $x$. Then, the choice of $\delta'$ implies
that $(f^{-1}(x_n))_{n\in\mathbb{Z}}$ can be $\epsilon'$-shadowed. So, the choice of $\epsilon'$ implies that
$(x_n)_{n\in\mathbb{Z}}$ can be $\epsilon$-shadowed. This ends the proof.
\end{proof}

The following lemma is proved as in Theorem 3.1.2 of \cite{ah}.

\begin{lem}
\label{l0}
If $f$ is a homeomorphism of a compact metric space, then
$$
CR(f)\cap Sh(f)\subset \Omega(f).
$$
\end{lem}

\begin{proof}
Fix $x\in Sh(f)\cap CR(f)$ and $\epsilon>0$.
For this $\epsilon$ we let $\delta$ be given by the shadowableness of $x$.
Since $x\in CR(f)$, there is a $\delta$-chain $\{x_i:0\leq i\leq k\}$ from $x$ to itself.
Define the sequence $\xi=(\xi_n)_{n\in\mathbb{Z}}$ by $\xi_{pk+i}=x_i$ for $p\in\mathbb{Z}$ and $0\leq i\leq k-1$.
It follows that $\xi$ is a $\delta$-pseudo-orbit through $x$ so
there is $y\in X$ such that $d(f^n(y),\xi_n)\leq\epsilon$ for every $n\in \mathbb{Z}$.
In particular, $d(y,x)\leq\epsilon$ and $d(f^k(y),x)\leq \epsilon$ and so
$f^k(B[x,\epsilon])\cap B[x,\epsilon])\neq\emptyset$. As $\epsilon$ is arbitrary, we get $x\in \Omega(f)$.
\end{proof}

As in theorems 2.3.3 and 2.3.4 of \cite{ah} we can prove the following.

\begin{lem}
\label{fk}
If $f: X\to X$ is a homeomorphism of a compact metric space $X$,
then $Sh(f)= Sh(f^k)$ for every $k\in\mathbb{Z}\setminus \{0\}$.
\end{lem}

\begin{proof}[Proof of Theorem \ref{thC}]
Let $X$ be a compact connected metric space with more than one point.
Suppose by contradiction that there is a minimal homeomorphism $f: X\to X$
of a compact connected metric space $X$ exhibiting a shadowable point $x$.
Fix $\epsilon>0$.
Clearly $x$ is nonwandering and so we can apply Lemma \ref{l2} to
obtain a point $y$ and a positive integer $k$ such that
the full orbit of $y$ under $f^k$ is contained in the $\epsilon$-ball $B[x,\epsilon]$.
However, $f$ is totally minimal (for the space is connected \cite{b})
so $f^k$ is minimal too.
It follows that the whole space is contained in $B[x,\epsilon]$.
Since $\epsilon$ is arbitrary, this implies that the space reduces to $x$ which is absurd.
This ends the proof.
\end{proof}

\begin{proof}[Proof of Theorem \ref{thB}]
We shall use the following facts about distal homeomorphisms $\phi: X\to X$ on compact metric spaces $X$:
Every $x\in X$ is {\em almost periodic}, i.e., for every neighborhood $U$ of $x$ there is a finite subset $F\subset \mathbb{Z}$
such that $\mathbb{Z}=F:\{n\in \mathbb{Z}:\phi^n(x)\in U\}$.
In particular, every $x$ is {\em recurrent} in the sense that $x\in \omega(x)$, where
$
\omega(p)=\{q\in X:q=\lim_{l\to\infty}\phi^{n_l}(p)\mbox{ for some sequence }n_l\to\infty\},
$ for every $p\in X$.
In particular, $\Omega(\phi)=X$.

To prove $Sh(f)\subset X^{deg}$ for distal homeomorphisms $f: X\to X$ we use
the argument in \cite{a} (or Theorem 11.5.5 of \cite{ah}) but with some modifications.
Take $z\in Sh(f)$ and suppose by contradiction that
$z\not\in X^{deg}$.
Then, the connected component $F$ of $X$ containing $z$ (which is compact)
has positive diameter $diam(F)>0$.
Take $0<\epsilon<\frac{1}{11}diam(F)$.
Since $f$ is distal, $z\in \Omega(f)$.
Then,
by Lemma \ref{l2}, there are $k\in \mathbb{N}^+$ and $y\in X$ such that $f^{nk}(y)\in B[z,\epsilon]$ for every $n\in\mathbb{Z}$.
Define $g=f^k$. Then,
\begin{equation}
\label{pelo}
g^n(y)\in B[z,\epsilon],\quad\quad\forall n\in\mathbb{Z}.
\end{equation}
On the other hand, $z\in Sh(g)$ by
Lemma \ref{fk}. Then, for the above $\epsilon$, we can choose $\delta>0$ from Lemma \ref{l1} with $K=\{z\}$.
We can assume that $\delta<\epsilon$.
Since $F$ is compact and connected, we can choose a sequence $y=p_1,p_2,\cdots, p_N\in F$ such that $d(p_i,p_{i+1})\leq\frac{\delta}{2}$
for $1\leq i\leq N-1$ and
\begin{equation}
\label{telo}
 F\subset \bigcup_{i=1}^NB[p_i,\delta].
\end{equation}

Since $z\in F$, we have $z\in B[p_{i_z}, \delta]$ for some $1\leq i_z\leq N$.
But $g$ is distal (for $f$ is) so every $p_i$ is recurrent with respect to $g$.
From this we can find positive integers $c(i)$ ($1\leq i\leq N$)
such that $d(p_i,g^{c(i)}(p_i))\leq \frac{\delta}{2}$ for all $i$.

As in \cite{a} we define the sequence $\eta$ by

\begin{eqnarray*}
\eta_i & = & g^i(p_1) \quad\quad \mbox{ if }i\leq0, \\
\eta_i & = & g^i(p_1) \quad\quad \mbox{ if }0\leq i\leq c(1)-1, \\
\eta_{c(1)+i} & = & g^i(p_2) \quad\quad \mbox{ if }0\leq i\leq c(2)-1, \\
& \vdots & \\
\eta_{c(1)+\cdots+c(N-2)+i} & = & g^i(p_{N-1}) \quad\quad \mbox{ if }0\leq i\leq c(N-1)-1, \\
\eta_{c(1)+\cdots+c(N-1)+i}& = & g^i(p_N) \quad\quad \mbox{ if }0\leq i\leq c(N)-1, \\
\eta_{c(1)+\cdots+c(N)+i} & = & g^i(p_{N-1}) \quad\quad \mbox{ if }0\leq i\leq c(N-1)-1, \\
& \vdots & \\
\eta_{c(1)+2\{c(2)+\cdots+c(N-1)\}+c(N)+i} & = & g^i(p_1) \quad\quad \mbox{ if }i\geq0.
\end{eqnarray*}

Clearly $\eta$ is a $\delta$-pseudo-orbit of $g$
and $\eta_{n_z}=p_{i_z}\in B[z,\delta]$, where
$$
n_z=c(1)+\cdots +c(i_z-1).
$$
Let $\xi$ be the sequence defined by $\xi_n=\eta_{n+n_z}$.
Clearly $\xi$ is a $\delta$-pseudo-orbit too but now through $B[z,\delta]$.
Then, by Lemma \ref{l1}, there is $x\in X$ such that $d(g^n(x),\xi_n)\leq \epsilon$ for every $n\in\mathbb{Z}$.
Then, by taking $\hat{z}=g^{-n_z}(x)$ we obtain
$d(g^n(\hat{z}),\eta_n)\leq\epsilon$ for every $n\in\mathbb{Z}$.
Since each $p_i\in \eta$ by definition, we conclude that there are
integers $n_1,\cdots, n_N$ satisfying
$d(g^{n_i}(\hat{z}),p_i)\leq \epsilon$ for every $1\leq i\leq N$.

Next we observe that by taking
$c=c(1)+2\{c(2)+\cdots+c(N-1)\}+c(N)$
we obtain
$$
d(g^i(\hat{z}),g^i(y))\leq\epsilon \quad(\mbox{for } i\leq c(1)-1)
\mbox{ and }d(g^{i+c}(\hat{z}),g^i(y))\leq \epsilon\quad(\mbox{for } i\geq0).
$$
This combined with (\ref{pelo}) yields
$$
g^i(\hat{z})\in B[z,2\epsilon] \mbox{ whenever } i\not\in ]c(1)-1,c[.
$$
However, $g$ is distal so $\hat{z}$ is recurrent with respect to $g$ thus
for every $j\in ]c(1)-1,c[$ there is $i_j\geq c$  such that $d(g^{i_j}(\hat{z}),g^j(\hat{z}))\leq \epsilon$.

Now take $w\in F$.
It follows from (\ref{telo}) that $d(w,p_i)\leq \delta$ for some $1\leq i\leq N$.
Then, $d(g^{j}(\hat{z}),w)\leq d(g^{j}(\hat{z}),p_i)+d(p_i,w)\leq \epsilon+\delta<2\epsilon$,
where $j=n_i$.
Now we have two cases:

If $j\not\in ]c(1)-1,c[$ then
$$
d(w,z)\leq d(w,g^j(\hat{z}))+d(g^j(\hat{z}),z)\leq 2\epsilon+2\epsilon=4\epsilon.
$$
If $j\in ]c(1)-1,c[$ then
$$
d(w,z)\leq d(w,g^j(\hat{z}))+d(g^j(\hat{z}),g^{i_j}(\hat{z}))+d(g^{i_j}(\hat{z}),z)
\leq 2\epsilon+\epsilon+2\epsilon=5\epsilon.
$$

From these cases we conclude that $F\subset B[z,5\epsilon]$ and so
$diam(F)\leq 10\epsilon$. But this contradicts the choice of $\epsilon$ so $z\in X^{deg}$.
As $z\in Sh(f)$ is arbitrary, we obtain $Sh(f)\subset X^{deg}$ and the proof follows.
\end{proof}

\begin{proof}[Proof of Theorem \ref{separated}]
Every equicontinuous homeomorphism $f: X\to X$ of a compact metric space $X$ is distal so
$Sh(f)\subset X^{deg}$ by Theorem \ref{thB}.
Conversely, take $p\in X^{deg}$ and $\epsilon>0$.

Since $f$ is equicontinuous, we can choose 
$\epsilon'>0$ 
such that
$r,s\in X$ and $d(r,s)\leq \epsilon'$ imply
$d(f^n(r),f^n(s))\leq \epsilon$ for all $n\in\mathbb{Z}$.

On the other hand, the proof of Proposition 3.1.7 in \cite{at} implies that
there is a {\em clopen} (i.e. open and closed subset)
$U$ of $X$ with diamenter $diam(U)\leq\epsilon'$ such that $p\in U$.
In particular, $dist(U,X\setminus U)>0$ and so we can take
$$
0<\delta'<\frac{dist(U,X\setminus U)}{2}.
$$
Again, since $f$ is equicontinuous, we can choose $\delta>0$ 
such that
$a,b\in X$ and $d(a,b)\leq \delta$ imply
$d(f^n(a),f^n(b))\leq \delta'$ for all $n\in\mathbb{Z}$.

Now take a $\delta$-pseudo-orbit $\xi$ with $\xi_0=p$ of $f$.

By definition $d(f(\xi_0),\xi_1)\leq \delta$ so
$d(f(p),\xi_1)\leq\delta$ thus $d(f^{n+1}(p),f^n(\xi_1))\leq\delta'$ for every $n\in\mathbb{Z}$. In particular,
$d(p,f^{-1}(\xi_1))\leq\delta'$ by taking $n=-1$.
Since $\{U, X\setminus U\}$ is a covering of $X$,
we have either $f^{-1}(\xi_1)\in U$ or $f^{-1}(\xi_1)\in X\setminus U$.
If $f^{-1}(\xi_1)\in X\setminus U$, then
$$
2\delta'< dist(X\setminus U,U)\leq d(p,f^{-1}(\xi_1))\leq \delta'
$$
which is absurd. Then, $f^{-1}(\xi_1)\in U$.

Again $d(f(\xi_1),\xi_2)\leq \delta$ so
$d(f^{n+1}(\xi_1),f^{n}(\xi_2))\leq\delta'$ for all $n\in\mathbb{Z}$.
In particular,
$d(f^{-1}(\xi_1),f^{-2}(\xi_2))\leq\delta'$ by taking $n=-2$.
Since $\{U, X\setminus U\}$ is a covering of $X$,
we have  either $f^{-2}(\xi_2)\in U$ or $f^{-2}(\xi_2)\in X\setminus U$.
If $f^{-2}(\xi_2)\in X\setminus U$, then
$$
2\delta'< dist(X\setminus U,U)\leq d(f^{-1}(\xi_1),f^{-2}(\xi_2))\leq\delta'
$$
which is absurd. Then, $f^{-2}(\xi_2)\in U$.

Repeating the argument we obtain $f^{-n}(\xi_n)\in U$ for every $n\in\mathbb{Z}$.

It follows  that $d(p,f^{-n}(\xi_n))\leq diam(U)\leq \epsilon'$ for every $n\in\mathbb{Z}$.
Hence, the choice of $\epsilon'$ implies
$d(f^k(p),f^{k-n}(\xi_n))\leq\epsilon$ for all $k,n\in \mathbb{Z}$.
By taking $k=n$ we obtain
$d(f^n(p),\xi_n)\leq \epsilon$ for all $n\in\mathbb{Z}$.
Then, $\xi$ can be $\epsilon$-shadowed (by the orbit of $p$).
Since $\epsilon$ is arbitrary, we obtain $p\in Sh(f)$
so $X^{deg}\subset Sh(f)$ as desired.
\end{proof}

From this lemma we obtain the following example.

\begin{ex}
\label{uno}
There are a compact metric space $X$ and a homeomorphism $f: X\to X$
such that $Sh(f)$ is a nonempty noncompact subset of $X$.
\end{ex}

\begin{proof}
Define $X=C\cup [1,2]$ with the topology induced from $\mathbb{R}$, where
$C$ be the ternary Cantor set of $[0,1]$.
Clearly $X^{deg}=C\setminus \{1\}$.
Now take $f: X\to X$ as the identity of $X$.
Since the identity is an equicontinuous homeomorphism, we obtain
$Sh(f)=X^{deg}$ by Theorem \ref{separated}. Then $Sh(f)=C\setminus \{1\}$.
Since $C\setminus \{1\}$ is nonempty and noncompact, we are done.
\end{proof}

\begin{remark}
\label{rm1}
It follows from Example \ref{uno} (with $K=C\setminus \{1\}$) that Lemma \ref{naosei} is false if $K$ were noncompact.
\end{remark}

\begin{proof}[Proof of Theorem \ref{thA}]
Let $f:X\to X$ be a homeomorphism of a compact metric space $X$.
We have that
$Sh(f)$ is invariant by Lemma \ref{invariant}.
There are homeomorphisms $f$ for which $Sh(f)$ is nonempty and noncompact by Example \ref{uno}.
By taking $K=X$ in Lemma \ref{naosei} we have that $f$ has the POTP if and only if
$Sh(f)=X$.
Finally, since $\Omega(f)\subset CR(f)$ we have that if $CR(f)\subset Sh(f)$ then $CR(f)=\Omega(f)$ by Lemma \ref{l0}.
This completes the proof.
\end{proof}

\end{document}